\documentclass[12pt]{article}

\makeatletter
\def\input@path{{styles/}}
\makeatother

\newcommand{\UsePackage}[1]{%
  \IfFileExists{styles/#1.sty}{%
      \usepackage{styles/#1}%
   }{%
      \IfFileExists{../styles/#1.sty}{%
         \usepackage{../styles/#1}%
      }{%
         \usepackage{#1}%
      }%
   }%
}

\usepackage[T1]{fontenc}
\usepackage{lmodern}
\usepackage{textcomp}

\usepackage{amsmath}%
\usepackage{amssymb}%
\usepackage[table]{xcolor}%

\usepackage{todonotes}
\usepackage[in]{fullpage}%

\usepackage[amsmath,thmmarks]{ntheorem}%
\theoremseparator{.}%

\usepackage{titlesec}%
\titlelabel{\thetitle. }%
\usepackage{xcolor}%
\usepackage{mleftright}%
\usepackage{xspace}%
\usepackage{graphicx}
\usepackage{hyperref}%
\usepackage{comment}

\usepackage{hyperref}%
\hypersetup{%
      unicode,
      breaklinks,%
      colorlinks=true,%
      urlcolor=[rgb]{0.25,0.0,0.0},%
      linkcolor=[rgb]{0.5,0.0,0.0},%
      citecolor=[rgb]{0,0.2,0.445},%
      filecolor=[rgb]{0,0,0.4},
      anchorcolor=[rgb]={0.0,0.1,0.2}%
}
\usepackage[ocgcolorlinks]{ocgx2}

\theoremseparator{.}%

\theoremstyle{plain}%
\newtheorem{theorem}{Theorem}[section]

\newtheorem{lemma}[theorem]{Lemma}

\newtheorem{corollary}[theorem]{Corollary}
\newtheorem{fact}[theorem]{Fact}

\newtheorem{proposition}[theorem]{Proposition}

\theoremstyle{plain}%
\theorembodyfont{\upshape}%
\newtheorem*{remark:unnumbered}[theorem]{Remark}%
\newtheorem{definition}[theorem]{Definition}

\newtheorem{example}[theorem]{Example}

\newcommand{\myqedsymbol}{\rule{2mm}{2mm}}

\theoremheaderfont{\em}%
\theorembodyfont{\upshape}%
\theoremstyle{nonumberplain}%
\theoremseparator{}%
\theoremsymbol{\myqedsymbol}%
\newtheorem{proof}{Proof:}%

\newtheorem{proofof}{Proof of\!}%

\providecommand{\emphind}[1]{}%
\renewcommand{\emphind}[1]{\emph{#1}\index{#1}}

\definecolor{blue25emph}{rgb}{0, 0, 11}

\providecommand{\emphic}[2]{}
\renewcommand{\emphic}[2]{\textcolor{blue25emph}{%
      \textbf{\emph{#1}}}\index{#2}}

\providecommand{\emphi}[1]{}%
\renewcommand{\emphi}[1]{\emphic{#1}{#1}}

\definecolor{almostblack}{rgb}{0, 0, 0.3}

\providecommand{\emphw}[1]{}%
\renewcommand{\emphw}[1]{{\textcolor{almostblack}{\emph{#1}}}}%

\providecommand{\emphOnly}[1]{}%
\renewcommand{\emphOnly}[1]{\emph{\textcolor{blue25}{\textbf{#1}}}}

\newcommand{\CaseyThanks}[1]{%
   \thanks{%
      Department of Mathematics; %
      University of Denver; %
      Clarence M. Knudson Building; %
      2390 South York Street; %
      Denver, CO, 80210, USA; %
      \href{mailto:casey.schlortt@du.edu}{casey.schlortt@du.edu} %
   #1%
   }%
}

\newcommand{\HLink}[2]{\hyperref[#2]{#1~\ref*{#2}}}
\newcommand{\HLinkSuffix}[3]{\hyperref[#2]{#1\ref*{#2}{#3}}}

\newcommand{\thmlab}[1]{{\label{theo:#1}}}
\newcommand{\thmref}[1]{\HLink{Theorem}{theo:#1}}

\newcommand{\explab}[1]{\label{exp:#1}}
\newcommand{\expref}[1]{\HLink{Example}{exp:#1}}%

\providecommand{\deflab}[1]{}
\renewcommand{\deflab}[1]{\label{def:#1}}
\newcommand{\defref}[1]{\HLink{Definition}{def:#1}}

\providecommand{\factlab}[1]{}
\renewcommand{\factlab}[1]{\label{fact:#1}}
\newcommand{\factref}[1]{\HLink{Fact}{fact:#1}}

\providecommand{\proplab}[1]{}
\renewcommand{\proplab}[1]{\label{prop:#1}}
\newcommand{\propref}[1]{\HLink{Proposition}{prop:#1}}

\newcommand{\lemlab}[1]{\label{lemma:#1}}
\newcommand{\lemref}[1]{\HLink{Lemma}{lemma:#1}}%

\providecommand{\eqlab}[1]{}%
\renewcommand{\eqlab}[1]{\label{equation:#1}}

\newcommand{\remove}[1]{}%

\usepackage[inline]{enumitem}

\newlist{compactenumA}{enumerate}{5}%
\setlist[compactenumA]{topsep=0pt,itemsep=-1ex,partopsep=1ex,parsep=1ex,%
   label=(\Alph*)}%

\newlist{compactenuma}{enumerate}{5}%
\setlist[compactenuma]{topsep=0pt,itemsep=-1ex,partopsep=1ex,parsep=1ex,%
   label=(\alph*)}%

\newlist{compactenumI}{enumerate}{5}%
\setlist[compactenumI]{topsep=0pt,itemsep=-1ex,partopsep=1ex,parsep=1ex,%
   label=(\Roman*)}%

\newlist{compactenumi}{enumerate}{5}%
\setlist[compactenumi]{topsep=0pt,itemsep=-1ex,partopsep=1ex,parsep=1ex,%
   label=(\roman*)}%

\newlist{compactitem}{itemize}{5}%
\setlist[compactitem]{topsep=0pt,itemsep=-1ex,partopsep=1ex,parsep=1ex,%
   label=\ensuremath{\bullet}}%

\numberwithin{figure}{section}%
\numberwithin{table}{section}%
\numberwithin{equation}{section}%

\newcommand{\N}{\mathbb{N}}

\newcommand{\Ob}{\mathcal{O}}
\newcommand{\A}{\mathcal{A}}

\begin{document}

\title{On the structure of sequences with minimal maximal pattern complexity}

\author{%
   Casey Schlortt%
   \CaseyThanks{}%
}

\date{\today}

\maketitle

\begin{abstract}
    In 2002, Kamae and Zamboni introduced maximal pattern complexity and determined that any aperiodic sequence must have maximal pattern complexity at least $2k$.  In 2006, Kamae and Rao examined the maximal pattern complexity of sequences over larger alphabets and showed that sequences which have maximal pattern complexity less than $\ell k$, for $\ell$ the size of the alphabet, must have some periodic structure. In this paper, we investigate the structure of sequences of low maximal pattern complexity over $\ell$ letters where $\liminf\limits\limits_{k \to \infty} p_{\alpha}^*(k) - 3k = -\infty$.  In addition, we show that the minimal maximal pattern complexity of an aperiodic sequence which uses all $\ell$ letters is $p_{\alpha}^*(k) = 2k + \ell -2$, and give an exact structure for aperiodic sequences with this maximal pattern complexity.
\end{abstract}

%%%%%%%%%%%%%%%%%%%%%%%%%%%%%%%%%%%%%%%%%%%%%%%%%%%%%%
%%%%%%%%%%%%%%%%%%%%%%%%%%%%%%%%%%%%%%%%%%%%%%%%%%%%%%

\section{Introduction}
There are various methods to study the complexity of symbolically defined dynamical systems.  Block (word) complexity, which counts the number of blocks of size $k$ which appear in a subshift, was introduced in 1938 by Morse and Hedlund as a way to determine the 'size' of different subshifts.  Their first paper, \cite{HM1}, contains the proof of the Morse-Hedlund Theorem which states that any aperiodic sequence must have block complexity at least $k+1$ for all $k \in \N$.  Their second paper, \cite{HM2}, introduces \textbf{Sturmian sequences}, defined to be infinite sequences over two letters which have block complexity $k+1$ for all $k \in \N$ and which can be represented as a coding of an irrational circle rotation.

For larger alphabets, a similar proof shows that the block complexity of an aperiodic sequence over $\ell \geq 2$ letters must be at least $k +\ell - 1$. A \textbf{Sturmian sequence  over $\ell \geq 2$ letters} is defined to be an aperiodic sequence with block complexity $k+\ell - 1$ for all $k \in \N$ \cite{Fogg}.

The maximal pattern complexity of an infinite sequence $\alpha$, denoted $p_{\alpha}^*(k)$, was introduced by Kamae and Zamboni in 2002 in \cite{Kamae_1} as a way to link word complexity and sequence entropy.  In their first paper, they showed that the maximal pattern complexity of any aperiodic sequence must be at least $2k$ for all $k \in \N$, which is a direct analogue to the  Morse-Hedlund Theorem.

In \cite{Kamae_1}, they also introduced \textbf{pattern Sturmian sequences} as an analogue to Sturmian sequences, and defined pattern Sturmian sequences over two letters to be aperiodic sequences with the lowest possible maximal pattern complexity, $2k$.  In \cite{Kamae_2}, they furthered their study of pattern Sturmian sequences, showing that sequences generated by a coding of an irrational circle rotation by any two interval partition of a circle and simple Toeplitz sequences are examples of pattern Sturmian sequences over two letters.  

In 2006, Kamae and Rao extended the study of maximal pattern complexity for sequences over finite alphabets of size greater than two in \cite{ell_let}. They determined that a sequence over $\ell \geq 2$ letters which has maximal pattern complexity less than $\ell k$ for some $k \in \N$ is \textbf{periodic by projection}, meaning there is a nonconstant, coordinatewise, factor map $f$ which maps the sequence to a periodic sequence.  They extended the definition of pattern Sturmian sequences over two letters to \textbf{pattern Sturmian sequences over $\boldsymbol{\ell \geq 2}$ letters}, defining them to be sequences of maximal pattern complexity $\ell k$ which are not periodic by projection. 

As stated in their paper, their results hold for sequences in which every letter appears infinitely often.  We will make this same assumption for all sequences in this paper, since the ability for a letter to show up a finite number of times can be eliminated with a finite shift.

Although \cite{ell_let} isolates the minimal maximal pattern complexity for aperiodic sequences which are not periodic by projection, there are aperiodic sequences with lower pattern complexity.  The goal of this work is to isolate the lowest maximal pattern complexity for any aperiodic sequence and determine the structure of such sequences.  We define a \textbf{strong pattern Sturmian sequence} over $\ell \geq 2$ letters to be an aperiodic sequence with maximal pattern complexity $2k + \ell - 2$.  As seen in \thmref{LB}, this is the lowest possible maximal pattern complexity an aperiodic sequence can have (given all letters of the alphabet appear in the sequence), and this notion aligns closely with the definition of a Sturmian sequence over $\ell$ letters.  Furthermore, since the notions of pattern Sturmian sequences and strong pattern Sturmian sequences coincide when the alphabet is two letters, we will restrict our study to sequences for which the size  of the alphabet is greater than two.

In this paper, we introduce three main results.  The first result, \thmref{main}, gives a general structure for aperiodic sequences with low maximal pattern complexity. \expref{4k_2aper} demonstrates that the bound of $4k$ in \thmref{main} cannot be substantially improved.

\begin{theorem}
    \thmlab{main}
    Let $\alpha$ be an aperiodic sequence such that $\#\A_{\alpha} = \ell \geq 3$.
    If 

    \begin{equation}
    \liminf\limits\limits_{k \to \infty} p_{\alpha}^*(k) - 3k = -\infty
    \end{equation}
    then there exists some decomposition of $\alpha$ into $m>1$ residues 
    \[\{\alpha^{(i)} = \alpha_i\alpha_{i+m}\alpha_{i+2m}... : 0 \leq i \leq m-1\}\]
    such that exactly one of these residues is eventually aperiodic over two letters, $a$ and $b$, and any residue containing finitely many $a$ and $b$ symbols is eventually constant.

    If $\alpha$ is uniformly recurrent, then the bound in Equation (1.1) can be improved to $4k$.
\end{theorem}

The second result, \thmref{S_C}, states that any strong pattern Sturmian sequence must consist of a pattern Sturmian sequence over two letters threaded together with constant sequences.  \expref{Lowest_Complexity} demonstrates the existence of sequences which are strong pattern Sturmian.

\begin{theorem}
    \thmlab{S_C}
    Let $\alpha$ be an aperiodic sequence such that $\#\A_{\alpha} = \ell \geq 3$.
     If $p_{\alpha}^*(k) = 2k + \ell -2$ for all $k \in \N$, then the aperiodic residue from \thmref{main} must be a pattern Sturmian over two letters and all other residues must be constant.
\end{theorem}

The third result, \thmref{CON}, is a partial converse of \thmref{S_C} which states that a sequence which can be decomposed into a pattern Sturmian sequence threaded together with constant sequences must have maximal pattern complexity $2k+C$ for some constant $C \in \N$.  \expref{Constant_Higher} demonstrates that a literal converse to \thmref{S_C} is untrue.  When $C=\ell - 2$, the sequence is strong pattern Sturmian.

\begin{theorem}
    \thmlab{CON}
    Let $\alpha$ be an aperiodic sequence such that $\#\A_{\alpha} = \ell \geq 3$.
    If $\alpha$ can be decomposed into residues such that one residue is a pattern Sturmian over two letters and all other residues are constant, then for sufficiently large $k$, $p_{\alpha}^*(k) = 2k + C$ for some $C \in \N$ with $C \geq \ell -2$.  
\end{theorem}

It is important to note that \thmref{main} allows for certain residues to be eventually periodic over two letters, but any nonconstant periodic residue must be periodic over the alphabet of the aperiodic residue.  However, \thmref{S_C} requires that any residue be aperiodic or constant.  Strong pattern Sturmian sequences also cannot have eventually constant residues and the aperiodic residue can only witness two letters. 

It is also important to recognize that in the case where $\ell = 3$, these results give a general structure for all aperiodic sequences $\alpha$ which have maximal pattern complexity less than $3k$. These sequences require a single aperiodic residue which is eventually over two letters, $a$ and $b$, threaded together with eventually constant sequences of the third letter $c$.

\section{Definitions and Preliminaries}
This section is dedicated to introducing definitions and results needed in the following sections.
\begin{definition}
    Let $\A$ be a finite set and $\alpha = \alpha_0\alpha_1\alpha_2...\in \A^{\N_0}$.  Then $\alpha$ is a  \textbf{(one-sided) infinite sequence over finite alphabet $\A$}.
\end{definition}

As mentioned in the introduction, we will only consider sequences in which a letter appearing once in the sequence means the letter appears infinitely in the sequence.

Throughout this paper, $\A^{\N_0}$ is endowed with the product topology.

\begin{definition}
    Let $\sigma: \A^{\N_0} \to \A^{\N_0}$ be the \textbf{shift map} defined by $\sigma(\alpha_0\alpha_1\alpha_2...) = \alpha_1\alpha_2...$.  We say that a \textbf{finite shift of $\boldsymbol{\alpha}$} is the sequence $\sigma^t(\alpha) = \alpha_t\alpha_{t+1}\alpha_{t+2}...$ for any integer $t \in \N_0$.  We say that the \textbf{orbit} of $\alpha$, denoted $\Ob(\alpha)$, is the set $\Ob(\alpha) :=\{\sigma^t(\alpha): t \in \N_0\}$.  The \textbf{orbit closure} of $\alpha$ is the orbit of $\alpha$ together with its limit points and is denoted $\overline{\Ob(\alpha)}$.
\end{definition}

    While $\alpha$ uses letters from $\A$, it may not use all letters from $\A$.  Thus to denote the letters which appear in an infinite sequence $\alpha$, we will use 
    
    $$\A_{\alpha} := \{a \in \A : \exists i\in \N_0 \text{ such that } \alpha_i = a\}.$$

\begin{definition}
    A sequence $\alpha$ is \textbf{periodic} if there exists some $t \in \N$ such that $\alpha = \sigma^t(\alpha)$. A sequence $\alpha$ is \textbf{eventually periodic} if there exists some $t \in \N$ such that $\sigma^t(\alpha)$ is periodic.  A sequence $\alpha$ is \textbf{aperiodic} if $\alpha$ is not periodic or eventually periodic.
\end{definition}

\begin{definition}
    A sequence $\alpha$ is \textbf{recurrent} if for all $L \in \N$ there exists some $M \in \N$ such that $\alpha_0...\alpha_{L-1} = \alpha_{M}...\alpha_{M+L-1}$.  If for all $L \in \N$, the set $R_L=\{M: \alpha_0...\alpha_{L-1} = \alpha_{M}...\alpha_{M+L-1}\}$ is syndetic (there are bounded gaps between subsequent elements), then $\alpha$ is \textbf{uniformly recurrent}.
\end{definition}

\begin{definition}
    A \textbf{$\boldsymbol{k}$-sized window $\boldsymbol{\tau}$} is an increasing set of $k$ indices beginning with 0: 
    $$\tau := \{0 = \tau(0) < \tau(1) < ... < \tau(k-1)\} \subset \N_0.$$ 

    We say a $(k+1)$-sized window $\tau'$ is an \textbf{immediate extension} of a $k$-sized window $\tau$ if $\tau' = \{\tau(0), \tau(1), ..., \tau(k-1), h\}$ for some $h > \tau(k-1)$.

    Given a $k$-sized window $\tau$, we will use the notation $m\tau$ to denote the $k$-sized window $m\tau : = \{m\tau(0), m\tau(1),..., m\tau(k-1)\}$.
\end{definition}

\begin{definition}
    A \textbf{$\boldsymbol{\tau}$-word over $\boldsymbol{\alpha}$} is a finite word of length $k$ denoted $\alpha[n + \tau]$ and defined to be
    $$\alpha[n + \tau] := \alpha_{\tau(0) + n}\alpha_{\tau(1) + n}...\alpha_{\tau(k-1) + n} \in \A^k$$
    for some $n \in \N_0$.

    The \textbf{set of all $\boldsymbol{\tau}$-words over $\boldsymbol{\alpha}$} is the set denoted by $F_{\alpha}(\tau)$ and defined by
    $$F_{\alpha}(\tau) := \{\alpha[n + \tau] : n \in \N_0\}.$$
\end{definition}

Since $F_{\alpha}(\tau)$ is finite, there exists some $L \in \N$ such that $F_{\alpha}(\tau) = \{\alpha[n + \tau] : n \in \N \cap [0, L)\}$.

\begin{definition}
    The \textbf{maximal pattern complexity} of an infinite sequence $\alpha$ is the function $p_{\alpha}^*:\N \to \N$ defined by 
    $$p_{\alpha}^*(k) := \sup_{\tau} \#F_{\alpha}(\tau)$$

    where the supremum is taken over all $k$-sized windows $\tau$.
\end{definition}

We note that for any $k \in \N$, there exists some $k$-sized window $\tau$ such that $p_{\alpha}^*(k) = \#F_{\alpha}(\tau)$ and we will call any such window $\tau$ a \textbf{maximal $\boldsymbol{k}$-sized window $\boldsymbol{\tau}$}.

In \cite{Kamae_1}, Kamae and Zamboni proved the following two theorems for maximal pattern complexity, the first which acts as an analogue to the Morse-Hedlund Theorem for block complexity and the second which relates the number of $\tau$-words to the number of $\tau'$-words in an aperiodic sequence $\alpha$ for some immediate extension $\tau'$ of $\tau$.

\begin{theorem}[\cite{Kamae_1}]
    \thmlab{KZ_HM}
    An infinite sequence $\alpha \in \A^{\N_0}$ for a finite alphabet $\A$ is eventually periodic if and only if $p_{\alpha}^*(k) \leq 2k - 1$ for some $k \in \N$.
\end{theorem}

\begin{theorem}[\cite{Kamae_1}]
    \thmlab{KZ_incr}
    For an arbitrary word $\alpha$ over a finite set $\A$, if there exists $k \in \N$ and a window $\tau$ of size $k$ such that $\#F_{\alpha}(\tau') \leq \#F_{\alpha}(\tau) + 1$ holds for any immediate extension $\tau'$ of $\tau$, then $\alpha$ is eventually periodic.
\end{theorem}

The following fact is an extension of the Morse-Hedlund Theorem for any window $\tau$, which follows from the same proof.

\begin{fact}
    \factlab{k+1}
    For any aperiodic sequence $\alpha$, any $k \in \N$, and any $k$-sized window $\tau$, $\#F_{\alpha}(\tau) \geq k+1$.
\end{fact}

Using \thmref{KZ_incr}, it is easy to extend \thmref{KZ_HM} to the case when $\A_{\alpha} = \A$ as seen in the following theorem.
\begin{theorem}
    \thmlab{LB}
    If an infinite sequence $\alpha$ is aperiodic and $\#\A_{\alpha} = \ell \geq 2$, then $p_{\alpha}^*(k) \geq 2k + \ell -2$ for all $k \in \N$.
\end{theorem}

\begin{proof}
    Suppose $\alpha$ is aperiodic and $\#\A_{\alpha} = \ell \geq 2$.  Clearly, since $\#\A_{\alpha} = \ell$, $p_{\alpha}^*(1) = \ell$.

    Suppose for some $k > 1$ we have $p_{\alpha}^*(k) \geq 2k + \ell -2$.

    Let $\tau$ be a maximal $k$-sized window for $\alpha$. Then $\#F_{\alpha}(\tau) \geq 2k + \ell - 2$.  By \thmref{KZ_incr}, since $\alpha$ is aperiodic, there exists some immediate extension $\tau'$ of $\tau$ such that $\#F_{\alpha}(\tau') \geq \#F_{\alpha}(\tau) + 2$.  Thus, 
    $$p_{\alpha}^*(k+1) \geq \#F_{\alpha}(\tau') \geq \#F_{\alpha}(\tau) + 2 \geq 2(k+1) + \ell -2.$$

    Hence for all $k \in \N$, $p_{\alpha}^*(k) \geq 2k + \ell -2$.
\end{proof}

In \cite{Kamae_1}, Kamae and Zamboni defined the maximal pattern complexity analogue for a Sturmian sequence over a two letter alphabet.

\begin{definition}
    A \textbf{pattern Sturmian sequence over a two letter alphabet} is an aperiodic sequence $\alpha$ for which $p_{\alpha}^*(k) = 2k$ for all $k \in \N$.  
\end{definition}

Kamae and Rao introduced the singular decomposition of an infinite sequence in \cite{ell_let}.  Below is a quick outline of this process, but a more in depth exploration can be found in \cite{ell_let}.

In \defref{SD_R}, the singular decomposition of a recurrent sequence is defined and in \defref{SD_NR}, the singular decomposition of a non-recurrent sequence is defined.  A singular decomposition requires a recurrent sequence, so for a non-recurrent sequence $\alpha$, we need to find some recurrent sequence $\beta$ in the orbit closure of $\alpha$ which satisfies some technical conditions outlined in \defref{Aux}, then decompose $\alpha$ with respect to the decomposition of $\beta$.  The following definitions outline the singular decomposition of any sequence $\alpha$.

\begin{definition}
    A sequence $\alpha$ is \textbf{periodic by projection} if there exists some proper, non-empty subset of the alphabet, $S \subset \A$, such that the sequence $1_{S}(\alpha_0)1_{S}(\alpha_1)1_{S}(\alpha_2)...\in \{0, 1\}^{\N_0}$ is eventually periodic.
    
    A letter $b \in \A_{\alpha}$ is \textbf{singular} if the sequence $1_{\{b\}}(\alpha_0)1_{\{b\}}(\alpha_1)1_{\{b\}}(\alpha_2)...\in \{0, 1\}^{\N_0}$ is periodic. 
\end{definition}

\begin{definition}
    Let $\beta$ be some recurrent sequence.  We say $m \in \N$ is the \textbf{decomposition cycle of $\boldsymbol{\beta}$} if $m$ is the smallest integer which is a period of $1_{\{b\}}(\alpha_0)1_{\{b\}}(\alpha_1)1_{\{b\}}(\alpha_2)...$ for each singular letter $b$ of $\beta$.  If $\beta$ has no singular letters, then we set $m=1$.
\end{definition}

\begin{definition}
    \deflab{SD_R}
    The \textbf{singular decomposition of a recurrent sequence} $\boldsymbol{\beta}$ is the set of sequences
    $$\{\beta^{(i)} = \beta_i\beta_{i+m}\beta_{i+2m}...: 0\leq i < m\}$$ 
    where $m$ is the decomposition cycle of $\beta$.
\end{definition}

\begin{definition}
    \deflab{Aux}
    Choose some $r \in \{0, 1, ..., m-1\}$ such that $\lim\limits_{n \to \infty} \sigma^{mk_n+r}(\alpha)$ converges for some sequence $(k_n)_{n\in\N}$.  Let $\beta = \sigma^{m-r}\left(\lim\limits_{n \to \infty} \sigma^{mk_n+r}(\alpha)\right)$.
    $\beta$ is an \textbf{auxiliary word of $\boldsymbol{\alpha}$} if the following conditions are true:
    \begin{enumerate}
        \item $\beta$ is recurrent
        \item $\beta \in \overline{\mathcal{O}(\alpha)}$
        \item $\beta^{(i)} \in \overline{\mathcal{O}(\alpha^{(i)})}$ for all $i \in \N \cap [0, m-1]$
    \end{enumerate}
\end{definition}

In \cite{ell_let}, Kamae and Rao proved that any aperiodic sequence has at least one auxiliary word.

Note that if $\beta$ is an auxiliary word of $\alpha$, then $p_{\alpha}^*(k) \geq p_{\beta}^*(k)$ for all $k \in \N$ since for any $k$-sized window $\tau$, $F_{\beta}(\tau) \subseteq F_{\alpha}(\tau)$.

\begin{definition}
    \deflab{SD_NR}
    The \textbf{singular decomposition of a non-recurrent word $\boldsymbol{\alpha}$ with respect to an auxiliary word $\boldsymbol{\beta}$} is the set of sequences
    $$\{\alpha^{(i)} = \alpha_i\alpha_{i+m}\alpha_{i+2m}...: 0\leq i < m\}$$ 
    for $m$ the decomposition cycle of the auxiliary word $\beta$.
\end{definition}

\begin{definition}
    Define an undirected graph $\Gamma$ on the residues of the singular decomposition of $\alpha$ by $V(\Gamma) = \{0, 1, 2, ..., m-1\}$, the residues of $\alpha$, and $E(\Gamma) = \{(i, j) : \A_{\alpha^{(i)}} \cap \A_{\alpha^{(i)}} \neq \emptyset\}$.
\end{definition}

\begin{definition}
    For any infinite sequence $\alpha$, $k$-sized window $\tau$, and residue $i$, $0 \leq i \leq m-1$ for $m$ the decomposition cycle of $\alpha$, we will define the \textbf{set of $\boldsymbol{\tau}$-words beginning on the $\boldsymbol{i}$th residue} as the set 
    $$C_i(\tau, \alpha) := \{\alpha[n + \tau] : n \equiv i \mod m\}$$
\end{definition}

We note that $\bigcup_{i=0}^{m-1} C_i(\tau, \alpha) = F_{\alpha}(\tau)$.

In \cite{ell_let}, Kamae and Rao also proved the following theorems about the maximal pattern complexity of sequences and found a threshold under which a non-trivial singular decomposition is forced.

\begin{theorem}[\cite{ell_let}]
    \thmlab{Rec_2.1}
    Let $\beta$ be a recurrent sequence over $\ell$ letters, $\ell \geq 2$.  If $p_{\beta}^*(k) < \ell k$ holds for some $k \in \N$, then $\beta$ contains at least one singular letter.
\end{theorem}

\begin{theorem}[\cite{ell_let}]
    \thmlab{NotRec_2.1}
    If $\ell \geq 2$ and $p_{\alpha}^*(k)< \ell k$ holds for some $k \in \N$, then $\Gamma$ is disconnected.
\end{theorem}

\begin{definition}[\cite{ell_let}]
    A sequence $\alpha$ over $\ell \geq 2$ letters is called a \textbf{pattern Sturmian sequence} if $p_{\alpha}^*(k) = \ell k$ and $\alpha$ is not periodic by projection.
\end{definition}

We define the following slightly more restrictive condition:

\begin{definition}
    A sequence $\alpha$ over $\ell \geq 2$ letters is called a \textbf{strong pattern Sturmian sequence} if $p_{\alpha}^*(k) = 2k + \ell -2$.
\end{definition}

Clearly, if $\ell = 2$, the notions of pattern Sturmian and strong pattern Sturmian are the same.

In general, two sequences may have very different maximal windows. 
 However, when both sequences are uniformally recurrent, the following theorem shows we can find a window which gives at least $2k$ words for each sequence.  This is a shared maximal window when both sequences are pattern Sturmian.  We prove this with an argument from the proof of \thmref{KZ_incr}, which can be found in \cite{Kamae_1}.
\begin{theorem}
    \thmlab{minimal}
    For any two aperiodic sequences $\alpha$ and $\beta$ which are uniformly recurrent and for each $k \in \N$, there exists some $k$-sized window $\tau$ such that $\#F_{\alpha}(\tau) \geq 2k$ and $\#F_{\beta}(\tau) \geq 2k$.
\end{theorem}

\begin{proof}
    Consider the product space $\overline{\Ob(\alpha)} \times\overline{\Ob(\beta)}$ with the map $\sigma \times \sigma$.  $\overline{\Ob(\alpha)} \times\overline{\Ob(\beta)}$ has a recurrent point (by a simple application of the Poincar\'e Recurrence Theorem), call it $(\gamma, \eta)$.

    Since $\alpha$ and $\beta$ are uniformly recurrent, $\gamma$ and $\eta$ are both aperiodic.

    Clearly, $\gamma$ and $\eta$ share a window of size one, and since they are both aperiodic, this window gives at least two words of length one for each sequence.

    Now suppose for some $k>1$ there exists a $(k-1)$-sized window $\tau$ such that $F_{\gamma}(\tau) \geq 2(k-1)$ and $F_{\eta}(\tau) \geq 2(k-1)$.

    Let $L \in \N$ such that $\gamma_0\gamma_1...\gamma_{L-1}$ and $\eta_0\eta_1...\eta_{L-1}$ witness all $\tau$ words over $\gamma$ and $\eta$ respectively.  Since $(\gamma, \eta)$ is recurrent, there exists $n \in \N$ such that for all $0 \leq i < L$, $(\gamma, \eta)_i = (\gamma, \eta)_{n + i}$.  Therefore, for all $0 \leq i < L$, $\gamma_i = \gamma_{i+ n}$ and $\eta_i = \eta_{i + n}$.  

    Define $\tau'$ to be an immediate extension of $\tau$ by $\tau'(i) = \tau(i)$ for all $0 \leq i \leq k-2$ and $\tau'(k-1) = n$.

    In the proof of \thmref{KZ_incr} from \cite{Kamae_1}, the following fact was proved: if $\delta$ is aperiodic and there is some $\tau' = \tau \cup \{n\}$ for some $n$ such that $awa \in \#F_{\delta}(\tau')$ for all $aw \in \#F_{\delta}(\tau)$, $a$ a single letter and $w$ a $k-2$-sized $\tau\setminus \{0\}$ word, then $\#F_{\delta}(\tau') \geq \#F_{\delta}(\tau) + 2$.

    Thus, for $\delta \in \{\gamma, \eta\}$, $\#F_{\delta}(\tau') \geq \#F_{\delta}(\tau) + 2$, and by the inductive hypothesis, $\#F_{\alpha}(\tau') \geq 2k$ and $\#F_{\beta}(\tau') \geq 2k$.
\end{proof}

\section{Recurrent Case}
The purpose of this section is to prove \thmref{main} in the case where $\alpha$ is recurrent.

For the entirety of this section, we will assume $\alpha$ is aperiodic, recurrent, $\#\A_{\alpha} = \ell \geq 3$, and $\liminf\limits\limits_{k \to \infty} p_{\alpha}^*(k) - 3k = -\infty$ (or in the case where $\alpha$ is uniformly recurrent, that $\liminf\limits\limits_{k \to \infty} p_{\alpha}^*(k) - 4k = -\infty$). 

By \thmref{Rec_2.1}, since $p_{\alpha}^*(k) < 3k \leq \ell k$ for some $k\in \N$, $\alpha$ has a singular letter.  Let $m$ be the decomposition cycle of $\alpha$, $m \geq 2$ since the decomposition is non-trivial by the existence of a singular letter.    

\begin{theorem}
    \thmlab{gen_lemR}
    If $\alpha$ is  aperiodic, recurrent, and $\liminf\limits\limits_{k \to \infty} p_{\alpha}^*(k) - 3k = -\infty$, then exactly one residue in the singular decomposition of $\alpha$ is aperiodic.  In the case where $\alpha$ is uniformly recurrent, the implication holds for $4k$ instead of $3k$.
\end{theorem}

\begin{proof}
    Suppose $\alpha$ is aperiodic, recurrent, and $\liminf\limits\limits_{k \to \infty} p_{\alpha}^*(k) - 3k = -\infty$.  Let $m$ be the decomposition cycle of $\alpha$, $m\geq 2$.
    
    Since $\alpha$ is aperiodic, there exists some $p$, $0 \leq p \leq m-1$, such that $\alpha^{(p)}$ is aperiodic. Let $k > m$ and let $\tau$ be a maximal $(k-(m-1))$-sized window over $\alpha^{(p)}$.  Then $\#F_{\alpha^{(p)}}(\tau) \geq 2k - 2(m-1)$ by \thmref{KZ_HM}.

    Define a $k$-sized window $\tau'$ over $\alpha$ by $\tau' = m\tau \cup \{h_i : 1 \leq i \leq m-1\}$ for $h_1 > m\tau(k-(m-1) -1)$, $h_{i+1} > h_i$ for all $1 \leq i \leq m-2$, and $h_i \equiv i \mod m$ for all $1 \leq i \leq m-1$.

    Define $x \in (\A\cup \{?\})^{\N}$ by $x_i = \alpha_i$ if $\alpha_i$ is a singular letter and $x_i = \,?$ if not.  Then $x$ has least period $m$ since a smaller least period for $x$ would result in the singular letters of $\alpha$ sharing a period smaller than $m$.

    Therefore, $C_i(\tau', \alpha)\, \cap \, C_j(\tau', \alpha) = \emptyset$ for all $0 \leq i < j \leq m-1$ since the position of the singular letters uniquely determines to which set $C_i(\tau', \alpha)$ a word belongs. Furthermore, $\#C_p(\tau', \alpha) \geq 2k-2(m-1)$ since every word in $F_{\alpha^{(p)}}(\tau)$ must appear as a prefix for some word in $C_p(\tau', \alpha)$.
    
    Suppose towards a contradiction that there exists an $i$, $0 \leq i \leq m-1$ and $i \neq p$, such that $\alpha^{(i)}$ is also aperiodic.  Consider $\tau$ over $\alpha^{(i)}$.  Since $\alpha^{(i)}$ is aperiodic, then by \factref{k+1}, $\#F_{\alpha^{(i)}}(\tau) \geq k - (m-1) +1$.  Thus, $\#C_i(\tau', \alpha) \geq k-(m-1)+1$, and since $C_p(\tau', \alpha)$ and $C_i(\tau', \alpha)$ are disjoint, $\#F_{\alpha}(\tau') \geq  3k-3(m-1) + 1$.  Therefore, 
    $$\liminf\limits_{k \to \infty} p_{\alpha}^*(k) - 3k \geq \liminf\limits_{k \to \infty} 3k-3(m-1)+1 - 3k \geq -3(m-1) + 1.$$   
    Since $m$ is fixed, this is a contradiction.  

    In the case where $\alpha$ is uniformly recurrent, any aperiodic residue must also be uniformly recurrent.  This is because the appearance of singular letters in any block of $\alpha$ forces the points where the block begins again to fall on exactly the same residue as the original block.  Thus, blocks in each residue must also appear syndetically, as they directly correspond to the uniform recurrence of the original sequence.  Therefore, if $\alpha^{(p)}$ and $\alpha^{(i)}$ are both aperiodic, then they are both uniformly recurrent, and by \thmref{minimal}, there exists a $(k-(m-1))$-sized window $\tau$ such that $\#F_{\alpha^{(p)}}(\tau) \geq 2k - 2(m-1)$ and $\#F_{\alpha^{(i)}}(\tau) \geq 2k - 2(m-1)$.  Define a $k$-sized window $\tau'$ over $\alpha$ by $\tau' = m\tau \cup \{h_i : 1 \leq i \leq m-1\}$ for $h_1 > m\tau(k-(m-1) -1)$, $h_{i+1} > h_i$ for all $1 \leq i \leq m-2$, and $h_i \equiv i \mod m$ for all $1 \leq i \leq m-1$.  Then $C_p(\tau', \alpha) \cap C_i(\tau', \alpha) = \emptyset$, $\#C_p(\tau', \alpha) \geq 2k - 2(m-1)$, and $\#C_i(\tau', \alpha) \geq 2k-2(m-1)$.  Therefore,
    $$\liminf\limits_{k \to \infty} p_{\alpha}^*(k) - 4k \geq \liminf\limits_{k \to \infty} 4k-4(m-1) - 4k \geq -4(m-1),$$ which is a contradiction since $m$ is fixed.
    
    Thus, exactly one residue of the singular decomposition of $\alpha$ is aperiodic.
\end{proof}

\begin{proposition}
    \proplab{rec_b}
    If $\alpha$ is  aperiodic, recurrent, and $\liminf\limits\limits_{k \to \infty} p_{\alpha}^*(k) - 3k = -\infty$, then the aperiodic residue from \thmref{gen_lemR} must be over two letters and for all $c \in \A$ if $c$ is not a letter in the aperiodic residue, $c$ is singular for $\alpha$.
\end{proposition}

\begin{proof}
    Let $\alpha^{(p)}$ be the aperiodic residue from \thmref{gen_lemR}, $0 \leq p \leq m-1$.  Suppose towards a contradiction that $\#\A_{\alpha^{(p)}} \geq 3$.  Since for any $k \in \N$ and $k$-sized window $\tau$ over $\alpha^{(p)}$, $F_{\alpha^{(p)}}(\tau) \subset F_{\alpha}(m\tau)$, then $p_{\alpha^{(p)}}^*(k) \leq p_{\alpha}^*(k) < 3k$. Thus by \thmref{Rec_2.1}, 
    $\alpha^{(p)}$ must have a singular letter, call it $c$. Since all other residues are periodic and $c$ appears in $\alpha^{(p)}$ periodically, $c$ must appear in $\alpha$ periodically, and so $c$ is singular for $\alpha$ and thus cannot appear in $\alpha^{(p)}$, a contradiction.  Thus $\#\A_{\alpha^{(p)}} = 2$.
    
    Suppose $c \in \A$, $c \notin \A_{\alpha^{(p)}}$.  Then $c$ only appears in residues which are periodic.  Thus $c$ appears periodically in $\alpha$, so $c$ is singular for $\alpha$.
\end{proof}

Clearly, \thmref{gen_lemR} and \propref{rec_b} together prove \thmref{main} in the case when $\alpha$ is recurrent.

\section{Non-Recurrent Case}

This section is dedicated to proving \thmref{main} in the non-recurrent case. For the entirety of this section, we assume that $\alpha$ is a non-recurrent, aperiodic sequence, $\#\A_{\alpha} = \ell \geq 3$, and $\liminf\limits\limits_{k \to \infty} p_{\alpha}^*(k) - 3k = -\infty$.

Let $\beta$ be an auxiliary word for $\alpha$. By \thmref{NotRec_2.1}, since $p_{\alpha}^*(k) < 3k \leq \ell k$ for some $k \in \N$, $\Gamma$ is disconnected.  Let $m$ be the decomposition cycle of the singular decomposition of $\alpha$ with respect to $\beta$, $m \geq 2$ since the decomposition is non-trivial.  

For each eventually periodic $\alpha^{(i)}$ there exists an $n_i \in \N_0$ such that $\sigma^{n_i}(\alpha^{(i)})$ is periodic.  For each aperiodic $\alpha^{(i)}$ there exists an $n_i \in \N_0$ such that each letter in $\A_{\sigma^{n_i}(\alpha^{(i)})}$ appears in $\sigma^{n_i}(\alpha^{(i)})$ infinitely often. Let $t = \max\{n_i : 0 \leq i \leq m-1\}$.  Then for the sequence $\alpha' = \sigma^{mt}(\alpha)$, all residues of $\alpha'$ are aperiodic or periodic, every letter appearing in a residue of $\alpha'$ appears in that residue infinitely often, $\alpha = w\alpha'$ for some finite word $w$, and $p_{\alpha}^*(k) = p_{\alpha'}^*(k) + C'$ for some constant $C' \in \N_0$ and $k$ sufficiently large (since for any window, a finite prefix can only contribute a finite number of words).  We further note that $\beta$ is an auxiliary word for $\alpha'$ and the singular decomposition of $\alpha'$ is $\{\alpha'^{(i)} = \sigma^{t}(\alpha^{(i)}) : 0\leq i\leq m-1\}$.  The graph $\Gamma'$ of the residues of $\alpha'$ is such that $V(\Gamma) = V(\Gamma')$ and $E(\Gamma') \subset E(\Gamma)$, and so $\Gamma'$ is also disconnected.

\begin{lemma}
    \lemlab{CC}
    Exactly one connected component of $\Gamma'$ contains aperiodic residues and all other connected components of $\Gamma'$ contain only singular residues.
\end{lemma}

\begin{proof}
    Since $\alpha'$ is aperiodic, at least one connected component of $\Gamma'$ contains residues which are aperiodic.
    
    Suppose towards a contradiction that there are two disjoint connected components of $\Gamma'$ which contain aperiodic residues.  Then there are two residues, $i \neq j$  for which $\alpha'^{(i)}$ and $\alpha'^{(j)}$ are both aperiodic on disjoint alphabets. Let $k \in \N$ and let $\tau$ be a maximal $k$-sized window over $\alpha'^{(i)}$.  Then by \thmref{KZ_HM}, $\#F_{\alpha'^{(i)}}(\tau) \geq 2k$.  Furthermore, $F_{\alpha'^{(i)}}(\tau) \cap F_{\alpha'^{(j)}}(\tau) = \emptyset$ since the alphabets of $\alpha'^{(i)}$ and $\alpha'^{(j)}$ are disjoint, and by \factref{k+1}, $\#F_{\alpha'^{(j)}}(\tau) \geq k+1$ since $\alpha'^{(j)}$ is aperiodic.  Then $\#F_{\alpha'}(m\tau) \geq 2k + k + 1$, a contradiction.

    Let $A \subset \{0, 1, ..., m-1\}$ be the set of residues in the aperiodic connected component and $S = \{0, 1, ..., m-1\} \setminus A$ be all other residues.

    Let $i \in S$.  Then $\alpha'^{(i)}$ is periodic, and so $\beta^{(i)}$ is also periodic.  Since the letters appearing in $\beta^{(i)}$ only appear in residues from $S$, those letters must appear periodically in $\beta$ and thus are singular for $\beta$.  Therefore, since the letters are singular for $\beta$ and periodic in $\alpha$, they must be singular for $\alpha'$, and for all $i \in S$, $\alpha'^{(i)}$ is constant.
\end{proof}

For the remainder of this section let $A$ be the set of residues in the aperiodic connected component and $S$ be the collection of all other residues.

\begin{lemma}
    \lemlab{gamma_stuff}
    If more than two letters appear in a single aperiodic residue, $\alpha'^{(i)}$, then one of the letters must appear periodically within $\alpha'^{(i)}$.
\end{lemma}
\begin{proof}
    Suppose $\alpha'^{(i)}$ is an aperiodic residue for $\alpha'$ and $\#\A_{\alpha^{(i)}} \geq 3$.  
    Then $p_{\alpha'^{(i)}}^*(k) \leq p_{\alpha'}^*(k) < 3k$ for sufficiently large $k$.

    If $\alpha'^{(i)}$ is recurrent, then by \thmref{Rec_2.1}, it must have some singular letter, and by definition, this letter is periodic in $\alpha'^{(i)}$.
    
    If $\alpha'^{(i)}$ is not recurrent, then by \thmref{NotRec_2.1}, the graph $\Gamma_{i}'$ of $\alpha'^{(i)}$ is disconnected, and by \lemref{CC}, only one connected component has aperiodic pieces and all other pieces are singular for $\alpha'^{(i)}$.  Thus there is at least one letter which appears periodically in $\alpha'^{(i)}$. 
\end{proof}

\begin{lemma}
    \lemlab{aper_let}
    If $c \in \A_A := \bigcup_{i \in A} \A_{\alpha'^{(i)}}$, then there exists some residue $j$, $0\leq j \leq m-1$ such that $c$ appears aperiodically in $\alpha'^{(j)}$.
\end{lemma}
\begin{proof}
    Let $c \in \A_A$.  If $c$ appears periodically in each $\alpha'^{(j)}$, $0 \leq j \leq m-1$ then $c$ would appear periodically in each $\beta^{(j)}$.  Thus $c$ would be singular for $\beta$, and since $c$ is periodic in every residue of $\alpha'$, $c$ would be singular for $\alpha'$, a contradiction.
\end{proof}

\begin{theorem}
    \thmlab{A_size}
    The alphabet of the aperiodic residues, $\A_A = \bigcup_{i \in A} \A_{\alpha'^{(i)}}$, has exactly two letters.
\end{theorem}
\begin{proof}
    Clearly $\#\A_A \geq 2$ or else no residue from $A$ would be aperiodic.  Thus, suppose towards a contradiction that $\#\A_A > 2$.

    Since $\beta$ is contained in the orbit closure of $\alpha'$, $p_{\beta}^*(k) \leq p_{\alpha'}^*(k)$.  Thus, since $\beta$ is recurrent, $\beta$ either has the form outlined in the conclusion of \thmref{main} by Section 3, or $\beta$ is periodic.

    Let $a, b, $ and $c$ be distinct letters in $\A_A$ and let $k > m$.
\newline
\newline
    \textbf{Case I:} Suppose $\beta$ has the form outlined in the conclusion of \thmref{main}. 
    
    Suppose $\beta^{(p)}$, $0 \leq p \leq m-1$, is the only aperiodic residue and is over the letters $\{a, b\}$, and all other residues of $\beta$ are either nonconstant periodic over $\{a, b\}$ or constant.
    
    Let $\tau$ be a maximal $k$-sized window over $\alpha'^{(p)}$.  Since $\beta$ is in the orbit closure of $\alpha'$, $F_{\beta^{(p)}}(\tau) \subset F_{\alpha'^{(p)}}(\tau)$, and since $\beta^{(p)}$ is aperiodic, $F_{\alpha'^{(p)}}(\tau)$ contains at least $2k$ $\tau$-words which only have $a$ and $b$ as letters.
    
    By \lemref{aper_let}, since $c \in \A_A$, there exists some residue $i$, $0\leq i \leq m-1$, such that $c$ appears aperiodically (and infinitely) in $\alpha'^{(i)}$.  Since $c \neq a, b$, then either $\beta^{(i)} = c^{\infty}$ or $c$ does not appear in $\beta^{(i)}$ because any nonconstant residue of $\beta$ can only be over the letters $a$ and $b$.  
    
    If $\beta^{(i)} = c^{\infty}$, then by the aperiodicity of $\alpha'^{(i)}$, there are $k-2$ $\tau$-words over $\alpha'^{(i)}$ of the form $c^jd_ju_j$, $0 < j< k$, $d_j \neq c$, and $u_j \in \A^{k-j-1}$.  Thus $p_{\alpha}^*(k) \geq \#F_{\alpha'}(m\tau) \geq 3k-2$, a contradiction.

    If $c$ does not appear in $\beta^{(i)}$, then there are at least $k-2$ $\tau$-words over $\alpha'^{(i)}$ with the form $u_jcu_j'$ for $0 < j < k$, and some $u_j \in (\A \setminus \{c\})^{j}$, $u_j' \in \A^{k-j-1}$. Thus $p_{\alpha}^*(k) \geq \#F_{\alpha'}(m\tau) \geq 3k-2$ for all $k > m$, a contradiction.
\newline
\newline
    \textbf{Case II:} Suppose $\beta$ is periodic.

    Since every letter of $\A_A$ must appear aperiodically in some residue of $\alpha'$ and $\#\A_A\geq 3$, by \lemref{gamma_stuff}, there must be at least two residues of $\alpha'$ which are aperiodic, $\alpha'^{(p)}$ and $\alpha'^{(i)}$, for some $0\leq p < i \leq m-1$.  Since $\beta$ is periodic, every letter in $\A_{\beta}$ is singular for $\beta$.  Then, without loss of generality, either $a^{\infty} = \beta^{(p)} \neq \beta^{(i)} = b^{\infty}$ or $a^{\infty} = \beta^{(p)} = \beta^{(i)}$.

    Let $\tau$ be a maximal $k$-sized window over $\alpha'^{(p)}$.
\newline
\newline
    \textbf{Case IIa:} Suppose that $a^{\infty} = \beta^{(p)} \neq \beta^{(i)} = b^{\infty}$.  Then $b$ must appear in $\alpha'^{(p)}$ infinitely often (and likewise $a$ must appear in $\alpha'^{(i)}$ infinitely often) or else there are at least $2k$ $\tau$-words in $\alpha'^{(p)}$ with no $b$ in them and $k-2$ $\tau$-words of the form $b^jdu_j$ for $0<j<k$, $d \neq b$, and $u_j \in \A^{k-j-1}$ in $\alpha^{(i)}$.  Thus $\#F_{\alpha}(m\tau) \geq 2k + k -2$, a contradiction.

    Since $c$ must appear in some residue of $\alpha'$ aperiodically, then either there exists some $q$, $0\leq q \leq m-1$ and $q \neq p,i$ such that $\beta^{(q)} = c^{\infty}$, or $c$ appears in $\alpha'^{(p)}$ and $\alpha'^{(i)}$ infinitely often (since appearing in one and not the other would give at least $2k$ words with no $c$ in one residue and $k-2$ words of the form $u_jcu_j'$ for $0<j<k$ in the other residue).

    If there exists some $q$, $0\leq q \leq m-1$ and $q \neq p,i$ such that $\beta^{(q)} = c^{\infty}$, then there are $k-2$ $\tau$-words of the form $a^jdu_j$ for $0<j<k$, $d \neq b$, and $u_j \in \A^{k-j-1}$  in $\alpha^{(p)}$; $k-2$ $\tau$-words of the form $b^jeu_j$ for $0<j<k$, $e \neq b$, and $u_j \in \A^{k-j-1}$  in $\alpha^{(i)}$; and $k-2$ $\tau$-words of the form $c^jfu_j$ for $0<j<k$, $f \neq c$, and $u_j \in \A^{k-j-1}$ in  in $\alpha^{(q)}$.  Thus $\#F_{\alpha}(m\tau) \geq 3k - 6$ for all $k > m$, a contradiction.

    Thus, $c$ appears in $\alpha'^{(p)}$ infinitely often.  Since $\alpha'^{(p)}$ sees $a, b,$ and $c$ infinitely often, by \lemref{gamma_stuff}, there is at least one letter in $\alpha'^{(p)}$ which appears periodically in $\alpha'^{(p)}$, a contradiction.
\newline
\newline
    \textbf{Case IIb:} Now assume that $\beta^{(p)} = \beta^{(i)} = a^{\infty}$.  Further assume that there are no residues $q$ for which $\alpha'^{(q)}$ is aperiodic and $\beta^{(q)} = d^{\infty}$ for $d \neq a$ or else we can reduce to Case IIa.

    Thus $b$ and $c$ must appear aperiodically in either $\alpha'^{(p)}$ or $\alpha'^{(i)}$ by \lemref{aper_let}.

    If, without loss of generality, $b$ only appears in $\alpha'^{(i)}$ and not in $\alpha'^{(p)}$, then there are at least $2k$ $\tau$-words from $\alpha'^{(p)}$ which do not contain a $b$, and $k-2$ $\tau$-words from $\alpha'^{(i)}$ of the form $u_jbu_j'$ for $0<j<k$, $u_j \in (\A\setminus\{b\})^j$ and $u_j' \in \A^{k-j-1}$.

    Thus $b$ and $c$ appear in $\alpha'^{(p)}$ infinitely often.  Then by \lemref{gamma_stuff}, there is some letter in $\alpha'^{(p)}$ which appears periodically, a contradiction.

    Therefore by cases I and II, $\#\A_A = 2$.
    \end{proof}  

\begin{theorem}
    \thmlab{gen_lemNR}
    There exists some decomposition of $\alpha'$ into $1< m_0 \leq m$ residues, where $m$ is the singular decomposition cycle of $\alpha'$, such that exactly one residue of the decomposition is aperiodic over two letters and any letter not in $\A_A$ is singular for $\alpha'$.
\end{theorem}
\begin{proof}
    Suppose that $\alpha'^{(p)}$, $0\leq p \leq m-1$, is aperiodic.  

    Let $k > m$ and let $\tau$ be a maximal $(k-(m-1))$-sized window over $\alpha'^{(p)}$.  Then $\#F_{\alpha'^{(p)}}(\tau) \geq 2k - 2(m-1)$ by \thmref{KZ_HM}.

    Define a $k$-sized window $\tau'$ over $\alpha'$ by $\tau' = m\tau \cup \{h_i : 1 \leq i \leq m-1\}$ for $h_1 > m\tau(k-(m-1) -1)$, $h_{i+1} > h_i$ for all $1 \leq i \leq m-2$, and $h_i \equiv i \mod m$ for all $1 \leq i \leq m-1$.

    Define $x \in (\A \cup \{?\})^{\N}$ by $x_i = \alpha'_i$ if $i \in S$ and $x_i = ?$ if $i \in A$.
    
    If $x$ has least period $m$, then $C_i(\tau', \alpha') \cap C_j(\tau', \alpha') = \emptyset$ for all $0 \leq i < j \leq m-1$. Furthermore, $\#C_p(\tau', \alpha') \geq 2k-2(m-1)$ since $C_p(\tau', \alpha')$ witnesses $\tau$ over $\alpha'^{(p)}$.
    
    If a second residue $\alpha'^{(i)}$, $1 \leq i \leq m-1$, were also aperiodic, then $\#F_{\alpha'^{(i)}}(\tau) \geq k - (m-1) +1$ by \factref{k+1}.  Thus $\#C_i(\tau', \alpha') \geq k-(m-1)+1$, and since $C_p(\tau', \alpha')$ and $C_i(\tau', \alpha')$ are disjoint, $\#F_{\alpha'}(\tau') \geq \#F_{\alpha'}(\tau') \geq 3k-3(m-1) + 1$ for all $k > m$, a contradiction. 

    If $x$ has least period $m_0 < m$, then $m_0$ divides $m$. Suppose that there is a second aperiodic residue, $\alpha'^{(i)}$.  By \lemref{CC}, $i \in A$.

    If $i \not\equiv p \mod m_0$, then $C_p(\tau', \alpha') \cap C_i(\tau', \alpha') = \emptyset$.  Thus, since $\#C_p(\tau', \alpha') \geq 2k - 2(m-1)$ and by \factref{k+1}, $\#C_i(\tau', \alpha') \geq k - (m-1) + 1$, then $\#F_{\alpha'}(\tau')\geq 3k - 3(m-1) + 1$, a contradiction.

    If $i \equiv p \mod m_0$ and for all $j \not\equiv 0 \mod m_0$, $\alpha'^{(j)}$ is periodic, and by \thmref{A_size}, $\#\A_A = 2$, then $\alpha'$ can be decomposed with respect to $m_0$.  We will denote this decomposition $\{\alpha'^{[i]} : 0 \leq i \leq m_0-1\}$. Then the only aperiodic residue is $\alpha'^{[p]}$ and all other residues either belong to $S$ and thus are singular for $\alpha'$, or are periodic over the alphabet of $\alpha'^{[p]}$. 

    Thus we have exactly one aperiodic residue over two letters.

    Suppose $c \notin \A_{\alpha'^{(p)}}$.  Then $c$ only appears in periodic residues.  Thus $c$ appears periodically in $\alpha'$ and thus must appear periodically in $\beta$.  Hence $c$ is singular for $\beta$ and must be singular for $\alpha'$.
    \end{proof}

By \thmref{gen_lemNR} and the discussion at the beginning of this section, the conclusion for \thmref{main} has been proved for when $\alpha$ is not recurrent, completing the proof of \thmref{main}.

\section{Strong Pattern Sturmian Sequences}
In this section we will prove \thmref{S_C} and \thmref{CON}.  First, \lemref{relating} and its immediate corollary introduce the relationship between the maximal pattern complexity of a sequence and the maximal pattern complexity of its aperiodic residue, which is unique by \thmref{main}.

\begin{lemma}
    \lemlab{relating}
    If $\liminf\limits_{k\to \infty}p_{\alpha}^*(k) - 3k = -\infty$ and $\alpha^{(p)}$ is the aperiodic residue from \thmref{main}, then $p_{\alpha}^*(k) \geq p_{\alpha^{(p)}}^*(k) + \ell - 2$.
\end{lemma}

\begin{proof}
    Let $k \in \N$ and let $\tau$ be a maximal $k$-sized window over $\alpha^{(p)}$.  Consider the window $m\tau$ over $\alpha$.  Then $\#F_{\alpha}(m\tau) \geq \#F_{\alpha^{(p)}}(\tau) + \ell - 2$ since there is at least one word beginning with each letter of the alphabet.  Thus $p_{\alpha}^*(k) \geq \#F_{\alpha}(m\tau) \geq p_{\alpha^{(p)}}^*(k) + \ell -2$.
\end{proof}

An immediate corollary of \lemref{relating} is the following:
\begin{corollary}
    If $p_{\alpha}^*(k) = 2k + C$ for some $C \geq \ell - 2$ and $\alpha^{(p)}$ is the only aperiodic residue, $0 \leq p \leq m-1$, then $p_{\alpha^{(p)}}^*(k) \leq 2k + D$ for some $0 \leq D \leq C - (\ell - 2)$ and all $k \in \N$.
\end{corollary}

\lemref{looks_good} sets the foundation for proving \thmref{S_C} by showing that a nontrivial prefix of a non-recurrent sequence, as discussed in the beginning of Section 4, cannot occur when the maximal pattern complexity is restricted to $2k + \ell -2$.

\begin{lemma}
    \lemlab{looks_good}
    If $p_{\alpha}^*(k) = 2k + \ell -2$ for all $k \in \N$, then the aperiodic residue from \thmref{main} is over two letters, $a$ and $b$, and for any $c \in \A$ which is not equal to $a$ or $b$, $c$ is singular.
\end{lemma}
\begin{proof}
    If $\alpha$ is recurrent, this is clearly true by \thmref{gen_lemR} and \propref{rec_b}.  Therefore, we will assume that $\alpha$ is not recurrent.  By the discussion at the beginning of Section 4, there is a word $w$ such that $\alpha = w\alpha'$, all residues of $\alpha'$ are aperiodic or periodic, and since $p_{\alpha'}^*(k) \leq p_{\alpha}^*(k)$, by \thmref{gen_lemNR}, $\alpha'$ is a sequence with exactly one aperiodic residue over two letters, $a$ and $b$, and for any $c \in \A$, $c \neq a, b$, $c$ is singular.  Thus it suffices to show that $w$ is trivial.

    Assume, towards a contradiction, that $w$ is non-trivial.

    Define $x \in (\A_{\alpha} \cup \{?\})^{\N_0}$ by $x_i = ?$ if $\A_{\alpha'^{(i)}} \subseteq \{a, b\}$ and $x_i = \alpha'_i$ if $\A_{\alpha'^{(i)}} \not\subseteq \{a, b\}$.  Without loss of generality, we can assume $x$ has least period $m$, or else $\alpha$ could be decomposed with respect to $m_0 < m$, and $\alpha = w'\alpha''$ where $w$ is still non-trivial and $\alpha''$ has the same structure as $\alpha'$.

    Let $\alpha'^{(p)}$, $0 \leq p \leq m-1$, be the only aperiodic residue of $\alpha'$.

    Since $w$ is non-trivial, there is some residue $i$ such that either there exists some $c \in \A_{\alpha^{(i)}}$ such that $c \notin \A_{\alpha'^{(i)}}$ and $\A_{\alpha^{(i)}} \neq \{a, b\}$ or $\A_{\alpha^{(i)}} = \{a, b\}$, $\alpha'^{(i)}$ is periodic (potentially constant), and for any $t \in \N$, $\sigma^{t}(\alpha^{(i)}) \neq \alpha^{(i)}$.
    
    Consider the first case, where there exists some $c \in \A_{\alpha^{(i)}}$ such that $c \notin \A_{\alpha'^{(i)}}$ and $\A_{\alpha^{(i)}} \neq \{a, b\}$.

    Let $k \geq 2$ and let $\tau$ be a maximal $k$-sized window over $\alpha'^{(p)}$. Then $F_{\alpha'}(m\tau) = F_{\alpha'^{(p)}}(\tau) \cup \{d^k : d \in \A_{\alpha} \text{ and } d \neq a, b\}$ and so $\#F_{\alpha'}(m\tau) \geq 2k + \ell -2$.  We also note that $F_{\alpha'^{(p)}}(\tau)$ only contains words over the letters $a$ and $b$.
    
    However, since there is some residue $i$ for which there exists some $c \in \A_{\alpha^{(i)}}$ such that $c \notin \A_{\alpha'^{(i)}}$, then either the $m\tau$-word $cd^{k-1}$ for $c \notin \{a, b\}$, $cd^{k-1}$ for $d \notin \{a, b\}$, or $cu$ for $c \notin \{a, b\}$ and $u \in F_{\alpha'^{(p)}}(\tau\setminus\{\tau(k-1)\})$, must also appear in $F_{\alpha}(m\tau)$, and so $\#F_{\alpha}(m\tau) \geq 2k + \ell - 1$, a contradiction.

    Now consider the second case, where $\A_{\alpha^{(i)}} = \{a, b\}$, $\alpha'^{(i)}$ is periodic (potentially constant), and for any $t \in \N$, $\sigma^{t}(\alpha^{(i)}) \neq \alpha^{(i)}$.  We will further assume that the first case cannot occur, so constant residues in $\alpha'$ over letters other than $a$ and $b$ must also be constant in $\alpha$.

    Let $q$ be the least period of $\alpha'^{(i)}$.  Since $m$ is the least period of $x$, there exists some $j$, $0 \leq j \leq m-1$, such that $\A_{\alpha^{(i + j \mod m)}} \cap \A_{\alpha^{(p + j \mod m)}} = \emptyset$.  Without loss of generality, suppose that $p + j \not\equiv i \mod m$ (or else use the distance $-j \mod m$ instead).  

    Define a $q$-sized window $\tau$ by $\tau = \{0, m, 2m, ..., (q-1)m\}$.  Define an immediate extension $\tau'$ of $\tau$ by $\tau' = \tau \cup \{h\}$ for some $h > (q-1)m$ and $h \equiv j \mod m$.  Then $C_i(\tau', \alpha') \cap C_p(\tau', \alpha') = \emptyset$ and $\{\sigma^t(\alpha'^{(i)})_0\sigma^t(\alpha'^{(i)})_1...\sigma^t(\alpha'^{(i)})_{q-1} : 0 \leq t \leq q-1\} \cup \{\alpha^{(i)}_0\alpha^{(i)}_1...\alpha^{(i)}_{q-1}\} \subset C_i(\tau, \alpha)$, so $\#C_i(\tau, \alpha) \geq q + 1$.

    If $i + j \equiv p \mod m$ then $\#C_i(\tau', \alpha) \geq q + 2$ since at least one word from $C_i(\tau, \alpha)$ must have at least two extensions in $C_i(\tau', \alpha)$ or else $\alpha^{(p)}$ would be eventually periodic.  Furthermore, $\#C_p(\tau', \alpha) \geq q+1$ by \factref{k+1} since $\alpha^{(p)}$ is aperiodic.  Since there are at least $\ell - 2$ residues which are constant over letters not equal to $a$ or $b$, there are at least $\ell - 2$ more $\tau'$-words.  Thus, $\#F_{\alpha}(\tau') \geq 2q + 2 + \ell - 1$, a contradiction.

    If $i + j \not\equiv p \mod m$, then $\#C_i(\tau', \alpha) \geq q + 1$, $\#C_p(\tau', \alpha) \geq q+1$ by \factref{k+1}.  Since $C_i(\tau, \alpha) \cap C_p(\tau, \alpha) = \emptyset$, either $\alpha^{(p-j\mod m)} = c^{\infty}$ or $\alpha^{(i-j\mod m)} = c^{\infty}$ for some $c \notin \{a, b\}$.  Thus, either $\#C_{i-j\mod m}(\tau', \alpha) \geq 2$ or $\#C_{p-j\mod m}(\tau', \alpha) \geq 2$.  Then there are at least $\ell -3$ $\tau'$-words over the residues which are constant over letters not equal to $a$, $b$, or $c$.  Thus, $\#F_{\alpha}(\tau') \geq 2q + 2 + \ell - 1$, a contradiction.

    Thus, $w$ must be trivial.
\end{proof}

We will now prove \thmref{S_C}.

\begin{proofof}{\textit{\thmref{S_C}:}}
    Suppose that $\alpha$ is an aperiodic sequence such that $\#\A_{\alpha} = \ell \geq 3$ and $p_{\alpha}^*(k) = 2k+\ell - 2$ for all $k \in \N$.  
    
    By \lemref{looks_good}, exactly one residue is aperiodic over a two letter alphabet $\{a, b\}$.  Let $\alpha^{(p)}$, $0 \leq p \leq m-1$, be the aperiodic residue.  Also by \lemref{looks_good}, for all $c \in \A$ such that $c \notin \A_{\alpha^{(p)}}$, $c$ is singular for $\alpha$. 

    By \lemref{relating}, since $2k+\ell-2 = p_{\alpha}^*(k) \geq p_{\alpha^{(p)}}^*(k) + \ell - 2$, then $p_{\alpha^{(p)}}^*(k) \leq 2k$.  Since $\alpha^{(p)}$ is aperiodic, by \thmref{KZ_HM}, $p_{\alpha^{(p)}}^*(k) \geq 2k$, and so the aperiodic residue must have maximal pattern complexity $2k$, and is a pattern Sturmian over two letters.

    Since each letter not witnessed by $\alpha^{(p)}$ is singular, any residue on letters other than those in $\A_{\alpha^{(p)}}$ are constant.  Thus it suffices to show that there is no residue $\alpha^{(j)}$, $j \neq p$, which is nonconstant periodic over $\A_{\alpha^{(p)}}$.  

    Suppose that there is some residue $j$, $0 \leq j \leq m-1$ and $j \neq p$, such that $\alpha^{(j)}$ is nonconstant periodic over $\A_{\alpha^{(p)}} = \{a, b\}$. Let $k > m$, $\tau$ be a maximal $(k - (m-1))$-sized window over $\alpha^{(p)}$, and define a $k$-sized window $\tau'$ over $\alpha$ by $\tau' = m\tau \cup \{h_i : 1 \leq i \leq m-1\}$ for $h_1 > m\tau(k-(m-1) -1)$, $h_{i+1} > h_i$ for all $1 \leq i \leq m-2$, and $h_i \equiv i \mod m$ for all $1 \leq i \leq m-1$. Then $\#C_p(\tau', \alpha) \geq 2k - 2(m-1)$.  Since $\alpha^{(j)}$ is nonconstant and $\alpha^{(p)}$ is aperiodic, in every $C_i(\tau', \alpha)$ there are at least three distinct words classified by the letters appearing in the parts of the window centered over $\alpha^{(p)}$ and $\alpha^{(j)}$: two distinct words with an $a$ or $b$ in the $\alpha^{(j)}$ slot, and a third word since if the letter appearing in the $\alpha^{(j)}$ slot always forced the letter appearing in the $\alpha^{(p)}$ slot, then $\alpha^{(p)}$ would be periodic (either constant, equal to $\alpha^{(j)}$, or the bit flip of $\alpha^{(j)}$). Therefore, $\#F_{\alpha}(\tau) \geq 2k - 2(m-1) + 3(m-1) = 2k + m - 1$.  Furthermore, $m \geq \ell$ since $\alpha^{(p)}$ and $\alpha^{(j)}$ only use the letters $a$ and $b$, and there are at least $\ell -2$ other constant residues over the letters not equal to $a$ or $b$.  Therefore, $\#F_{\alpha}(\tau) \geq 2k + \ell - 1$, a contradiction.     
\end{proofof}

The next part of this section is dedicated to proving \thmref{CON}.
\begin{proofof}{\textit{\thmref{CON}:}}
    Suppose $\alpha$ is aperiodic, $\#\A_{\alpha} = \ell \geq 3$, and there exists some $m \geq 2$ such that the decomposition of $\alpha$ with respect to $m$ has one residue which is a pattern Sturmian over two letters and all other residues are constant.

    Let $\alpha^{(p)}$, $0 \leq p \leq m-1$, be the aperiodic pattern Sturmian over two letters.  

    Let $k \in \N$ and let $\tau$ be any $k$-sized window.  Define $\tau_i = \{\tau(j) : j \equiv i \mod m\}$ for all $0 \leq i \leq m-1$.  Then $\tau = \tau_0 \sqcup \tau_1 \sqcup ... \sqcup \tau_{m-1}$.  

    Suppose that $I = \{i: 0 \leq i \leq m-1 \text{ and } \tau_i \neq \emptyset\}$.

     Since any word in $C_i(\tau, \alpha)$ is completely determined by the letters of $\alpha^{(p)}$ which appear in it, $\#C_i(\tau, \alpha) \leq 2(\#\tau_{p-i \mod m})$.

    Furthermore, $\#\left(\bigcup_{i \notin I} C_{p-i \mod m}(\tau, \alpha)\right)\leq m-1-\#I$ since for all $i \notin I$, $C_{p-i \mod m}(\tau, \alpha)$ is a set of words which only see constant sequences, so each $C_{p-i \mod m}(\tau, \alpha)$ can contribute at most one word.

    Therefore, for all $k$-sized windows $\tau$, 
    $$\#F_{\alpha}(\tau) \leq \sum_{i \in I} 2(\#\tau_i) + \#\left(\bigcup_{i \notin I} C_i(\tau, \alpha)\right) \leq 2k + m-1-\#I \leq 2k + m-2.$$  
    Thus $p_{\alpha}^*(k) \leq 2k + m-2$.  By an application of \thmref{KZ_incr}, $p_{\alpha}^*(k+1) \geq p_{\alpha}^*(k) + 2$ for all $k \in \N$, and so for sufficiently large $k$, $p_{\alpha}^*(k) = 2k+C$ for some $\ell-2 \leq C \leq m-2$.
\end{proofof}

\section{Examples}

This section introduces some relevant examples.  \expref{Lowest_Complexity} demonstrates the existence of strong pattern Sturmian sequences over any finite alphabet.

\begin{example}
    \explab{Lowest_Complexity}
    We will define a sequence $\alpha \in \{0, 1, ..., \ell-1\}^{\N_0}$ which has maximal pattern complexity $p_{\alpha}^*(k) = 2k + \ell - 2$.

    Let $x \in \{0, 1\}^{\N_0}$ be a pattern Sturmian sequence over two letters.  Define $\alpha \in \{0, 1, ..., \ell-1\}^{\N_0}$ by 
    \[
    \alpha_{(\ell - 1)i + j} = \begin{cases}
        x_i & \text{if } j \equiv 0 \mod \ell -1\\
        j + 1 & \text{if } j\not\equiv 0 \mod \ell -1.
    \end{cases}
    \]
    
    The singular decomposition cycle of $\alpha$ is $\ell - 1$.

    Let $k \in \N$ and let $\tau$ be any $k$-sized window. Define $\tau_i = \{\tau(j) : j \equiv i \mod m\}$ for all $0 \leq i \leq m-1$.  Then $\tau = \tau_0 \sqcup \tau_1 \sqcup ... \sqcup \tau_{m-1}$. 

    Suppose that $I = \{i : 0 \leq i \leq \ell -2 \text{ and } \tau_i \neq \emptyset\}$.

    Since any two residues have disjoint alphabets, for all $0 \leq i < j \leq \ell - 2$, $C_i(\tau, \alpha)\cap C_j(\tau, \alpha) = \emptyset$.

    Furthermore, $C_i(\tau, \alpha) \leq \max\{2(\#\tau_{\ell -1 -i}), 1\}$, and so $\#F_{\alpha}(\tau) = \sum_{i = 0}^{\ell-2}\#C_i(\tau, \alpha) \leq 2k + (\ell - 1 - \#I) \leq 2k + \ell - 2$ since $\#I \geq 1$.

    Since this is true for any $k$-sized window $\tau$, we have that $p_{\alpha}^*(k) \leq 2k + \ell -2$ and since by \thmref{LB}, $p_{\alpha}^*(k) \geq 2k + \ell -2$, for this sequence, $p_{\alpha}^*(k) = 2k + \ell -2$.  
\end{example}

\expref{Constant_Higher} is a sequence which has the structure outlined in \thmref{CON} but has maximal pattern complexity strictly greater than $2k+\ell -2$ demonstrating that the converse of \thmref{S_C} is not true.

\begin{example}
    \explab{Constant_Higher}
    We will define a sequence $\alpha \in \{0, 1, 2, 3\}^{\N_0}$ which is a pattern Sturmian sequence zippered with constant sequences, but has maximal pattern complexity $p_{\alpha}^*(k) = 2k + C$ for $C > \ell - 2$.

    Let $x \in \{0, 1\}^{\N_0}$ be a pattern Sturmian over two letters.  Define $\alpha$ by
    \[
    \alpha_{5i + j} = \begin{cases}
        x_i & \text{if } j \equiv 0 \mod 5\\
        2 & \text{if } j\equiv 1, 4 \mod 5\\
        3 &\text{if } j \equiv 2, 3 \mod 5.
    \end{cases}
    \]

    Let $\tau$ be a maximal window for $x$ of size $k - 1$.  Define $\tau'$ by $\tau'(i) = m\tau(i)$ for all $0 \leq i \leq k-2$ and $\tau'(k-1) = h$ for some $h \equiv 1 \mod 5$ and $h > m\tau(k-2)$.

    Then $\#C_0(\tau', \alpha) = 2(k-1)$ since $\alpha^{(0)}$ is a pattern Sturmian sequence over two letters and $\tau$ is a maximal window for $\alpha^{(0)}$.  Furthermore, $\#C_4(\tau, \alpha) = 2$ since $C_4(\tau, \alpha) = \{2^{k-1}0, 2^{k-1}1\}$, and $\#C_1(\tau, \alpha) = \#C_2(\tau, \alpha) = \#C_3(\tau, \alpha) = 1$ since $C_1(\tau, \alpha) = \{2^{k-1}3\}$, $C_2(\tau, \alpha) = \{3^k\}$, and $C_3(\tau, \alpha) = \{3^{k-1}2\}$.  Clearly $\{C_i(\tau', \alpha) : 0 \leq i \leq 4\}$ are pairwise disjoint.  Thus, $\#F_{\alpha}(\tau') = 2k + 3 > 2k + (4-2) = 2k + \ell - 2$ for $\ell = 4$.
    
\end{example}

\expref{4k_2aper} is a uniformly recurrent sequence with two aperiodic residues which has complexity $4k$, showing that the bound in \thmref{main} cannot be substantially improved below $4k$.

\begin{example}
    \explab{4k_2aper}
    We will create a sequence which has complexity $4k$ and is two recurrent pattern Sturmian sequences zippered together with a constant sequence.

    Let $x \in \{0, 1\}^{\N_0}$ be the Fibonacci substitution defined by $0 \to 01$ and $1 \to 0$, a Sturmian and therefore pattern Sturmian sequence, which is also uniformly recurrent.

    Define $\alpha \in \{0, 1, 2\}^{\N_0}$ by

    \[
    \alpha_{3i + j} = \begin{cases}
        x_i & \text{if } j \equiv 0, 1 \mod 3\\
        2 &\text{if } j \equiv 2 \mod 3.
    \end{cases}
    \]

    First, we must show that $\alpha$ is uniformly recurrent. Let $L > 0$ and let $R_L(\alpha) = \{M \in \N : \alpha_0...\alpha_{L-1} = \alpha_{M}...\alpha_{M+L-1}\}$.  Let $L' = \lceil\frac{L}{3}\rceil$ and define $R_L'(x) = \{M \in \N : x_0...x_{L'-1} = x_{M}...x_{M+L'-1}\}$.  Since $x$ is uniformly recurrent, $R_{L'}(x) = \{M_1, M_2, ...\}$ is syndetic and for any $M_i$, $x_0x_1...x_{L'-1} = x_{M_i}x_{M_i + 1}...x_{M_i + L'-1}$.  Thus, $x_0x_02x_1x_12...x_{L'-1}x_{L'-1}2 = x_{M_i}x_{M_i}2x_{M_i + 1}x_{M_i + 1}2...x_{M_i + L'-1}x_{M_i + L'-1}2$ and so $R_L(\alpha) = \{3M_1, 3M_2, ...\}$ is syndetic and thus $\alpha$ is a uniformly recurrent sequence.

    Now we will show that $\alpha$ has maximal pattern complexity $p_{\alpha}^*(k) = 4k$ for all $k \geq 2$.

    First, we show that for any $k$-sized window $\tau$ over $\alpha$, $p_{\alpha}^*(k) \leq 4k$.

     Let $k \in \N$ and let $\tau$ be any $k$-sized window over $\alpha$. If for all $0 \leq i \leq k-1$, $\tau(i) \equiv 0 \mod 3$ (ie $\tau$ only sees one residue at a time), then $C_0(\tau, \alpha) = C_1(\tau, \alpha)$ since $\alpha^{(0)} = \alpha^{(1)} = x$, and $\#C_0(\tau, \alpha) \leq 2k$ since $x$ is a pattern Sturmian sequence.  Furthermore, $C_2(\tau, \alpha) = \{2^k\}$.  Thus, $\#F_{\alpha}(\tau) \leq 2k + 1$.
    
    Suppose now that $\tau = \tau_1 \sqcup \tau_2 \sqcup \tau_3$ for $\tau_i = \{\tau(j): \tau(j) \equiv i \mod 3\}$ and there exists $i\neq j$ such that $\#\tau_i, \#\tau_j > 0$. Then, for all $0 \leq i < j \leq 2$, $C_i(\tau, \alpha) \cap C_j(\tau, \alpha) = \emptyset$ since the position of the $2$s uniquely determines to which set $C_i(\tau, \alpha)$ a $\tau$-word belongs.

    Then $\#C_0(\tau', \alpha) \leq 2(\#\tau_0) + 2(\#\tau_1)$, $\#C_1 \leq 2(\#\tau_0) + 2(\#\tau_2)$ and $\#C_2 \leq 2(\#\tau_1) + 2(\#\tau_2)$.  Therefore 
    $$\#F_{\alpha}(\tau) \leq 2(\#\tau_0) + 2(\#\tau_1) + 2(\#\tau_0) + 2(\#\tau_2) + 2(\#\tau_1) + 2(\#\tau_2) = 4k $$

    Thus for all $k \in \N$ and any $k$-sized window $\tau$, $\#F_{\alpha}(\tau) \leq 4k$, and so $p_{\alpha}^*(k) \leq 4k$ for all $k \in \N$.

    Now we will show that for all $k \geq 2$ there exists some window $\tau'$ such that $\#F_{\alpha}(\tau') = 4k$.

    Let $\tau$ be a maximal $k$-sized window for $x$ such that the immediate restriction of $\tau$ is a maximal $(k-1)$-sized window.  Define $\tau'$ by $\tau'(i) = 3\tau(i)$ for all $0 \leq i \leq k-2$ and $\tau'(k-1) = 3\tau(k-1) + 1$.

    For all $0 \leq i < j \leq 2$, $C_i(\tau, \alpha) \cap C_j(\tau, \alpha) = \emptyset$ since words in $C_0(\tau', \alpha)$ are of the form $w_0$ for $w_0 \in \{0, 1\}^{k}$, words in $C_1$ are of the form $w_12$ for $w_1 \in \{0, 1\}^{k-1}$, and words in $C_2$ are of the form $2^{k-1}w_3$ for $w_3 \in \{0, 1\}$.
    
    Then $\#C_0(\tau', \alpha) = 2k$ since $C_0(\tau', \alpha) = F_{x}(\tau)$, $\#C_1(\tau', \alpha) = 2k-2$ since $C_1(\tau', \alpha) = \{w_12 : w_1 \in F_{x}(\tau\setminus \{\tau(k-1)\})\}$, and $\#C_2(\tau', \alpha) = 2$ since $C_2(\tau', \alpha) = \{2^{k-1}0, 2^{k-1}1\}$.

    Thus, $\#F_{\alpha}(\tau') = 4k$ and so $p_{\alpha}^*(k) = 4k$ for all $k \in \N$.
\end{example}

\section*{Acknowledgments}
The author would like to express her gratitude to her advisor Prof. Ronnie Pavlov for suggesting this question and for the numerous helpful comments, suggestions, and conversations during the formation of this paper.

%-------------------------------------------------------------------------

\bibliographystyle{alpha}
\bibliography{bib}

\begin{thebibliography}{KZ02b}

\bibitem[Fog02]{Fogg}
N.~Pytheas Fogg.
\newblock {\em Substitutions in dynamics, arithmetics and combinatorics},
  volume 1794 of {\em Lecture Notes in Mathematics}.
\newblock Springer-Verlag, Berlin, 2002.

\bibitem[KR06]{ell_let}
Teturo Kamae and Hui Rao.
\newblock Maximal pattern complexity of words over {$l$} letters.
\newblock {\em European J. Combin.}, 27(1):125--137, 2006.

\bibitem[KZ02a]{Kamae_2}
Teturo Kamae and Luca Zamboni.
\newblock Maximal pattern complexity for discrete systems.
\newblock {\em Ergodic Theory Dynam. Systems}, 22(4):1201--1214, 2002.

\bibitem[KZ02b]{Kamae_1}
Teturo Kamae and Luca Zamboni.
\newblock Sequence entropy and the maximal pattern complexity of infinite
  words.
\newblock {\em Ergodic Theory Dynam. Systems}, 22(4):1191--1199, 2002.

\bibitem[MH38]{HM1}
Marston Morse and Gustav~A. Hedlund.
\newblock Symbolic dynamics.
\newblock {\em American Journal of Mathematics}, 60(4):815--866, 1938.

\bibitem[MH40]{HM2}
Marston Morse and Gustav~A. Hedlund.
\newblock Symbolic dynamics {II}. {S}turmian trajectories.
\newblock {\em American Journal of Mathematics}, 62(1):1--42, 1940.

\end{thebibliography}

\end{document}